\documentclass{amsart}

\usepackage[title]{appendix}

\usepackage[french,english]{babel}
\usepackage[utf8]{inputenc}
\usepackage[T1]{fontenc}

\usepackage{hyperref}

\usepackage{comment}

\usepackage{palatino}

\usepackage[autostyle]{csquotes}

\usepackage{amsmath}
\usepackage{amsfonts,amssymb,amsmath,latexsym,mathrsfs,bbm,stmaryrd}
\usepackage[all]{xy}
\usepackage[usenames]{color}
\usepackage{epsfig}

\usepackage{graphicx}
\usepackage{subcaption}
\usepackage{float}

\usepackage[shortlabels]{enumitem}

\usepackage{amsthm}
\usepackage{thmtools}
\usepackage{thm-restate}
\declaretheorem[numberwithin=section]{theorem}
\declaretheorem[sibling=theorem]{proposition}
\declaretheorem[sibling=theorem]{definition}

\declaretheorem[sibling=theorem]{lemma}

\declaretheorem[sibling=theorem]{notation}

\declaretheorem[sibling=theorem,style=remark]{remark}
\declaretheorem[sibling=theorem,style=remark]{example}

\numberwithin{equation}{section}

\usepackage[mathcal]{euscript}
\usepackage{times}

\def\R{\mathbb R}
\def\C{\mathbb C}

\def\d{\delta}

\def\1{\mathbbm{1}}

\def\adots{
	\mathinner{\mkern1mu\raise1pt\hbox{.}\mkern2mu\raise4pt\hbox{.}
		\mkern2mu\raise7pt\vbox{\kern7pt\hbox{.}}\mkern1mu}}

\long\def\symbolfootnote[#1]#2{\begingroup%
\def\thefootnote{\fnsymbol{footnote}}\footnote[#1]{#2}\endgroup}

\def\C{{\mathbb C}}

\def\R{{\mathbb R}}

\def\CD{{\mathcal D}}

\def\CM{{\mathcal M}}
\def\CN{{\mathcal N}}

\def\CV{{\mathcal V}}

\def\unit{{\bf 1}}
\def\i{{\rm i}}

\def\d{{\rm d}}

\def\<{{\langle}}
\def\>{{\rangle}}

\def\CMD{\CM^\Delta}

\def\cc{^*}

\begin{document}

\title{Brown measure of $R$-diagonal operators, revisited}
\author{
	PING ZHONG \\
}

\address{
	Ping Zhong \\
	Department of Mathematics and Statistics\\
	University of Wyoming\\
	Laramie, WY 82070, USA\\
	{pzhong@uwyo.edu}
}

\thanks{Supported in part by Collaboration Grants for Mathematicians from Simons Foundation} 

\date{}
\maketitle

\begin{abstract}
We use subordination functions perspective to reformulate Haagerup--Schultz's approach for the Brown measure of $R$-diagonal operators. This allows us to simplify the original argument and find a connection with the other approach due to Belinschi-\'{S}niady-Speicher. The Brown measure formula can be rewritten in terms of subordination functions.
\end{abstract}

\section{Introduction}
The Brown measure of an operator in a finite von Neumann algebra with faithful, normal, tracial
state is an extension of eigenvalue counting measure of finite dimensional square matrices. In light of Voiculescu's breakthrough work \cite{Voiculescu1991} on asymptotic freeness of Gaussian random matrices, the Brown measure of a free random variable is a natural candidate of the limit of empirical spectral distribution of suitable random matrix model. Let $(X_n)_{n=1}^\infty$ be a sequence of $n\times n$ matrices whose entries are i.i.d. copies of a complex random variable with mean zero and unit variance.
The circular law \cite{Bordenave-Chafai-circular} says that the limit empirical spectral distribution of $X_n/\sqrt{n}$  is the Brown measure of Voiculescu's circular operator. The limit distribution in the single ring theorem \cite{Guionnet2011-singlering} is the Brown measure of $R$-diagonal operators, a large class of operators in the framework of free probability theory including circular operator. 

The calculation of explicit Brown measure formula is technical in general and only some special class of operators whose explicit Brown measure formulas are known up to now. The first nontrivial example was given by Haagerup--Larsen \cite{HaagerupLarsen2000} in the year 2000, almost twenty years after L.G.Brown introduced his measure, where they calculated the Brown measure of all bounded $R$-diagonal operators. Haagerup--Schultz \cite{HaagerupSchultz2007} extended the Brown measure formula for unbounded $R$-diagonal operators, and such work played an important role in their seminar work \cite{HaagerupSchultz2009} on invariant subspace problem in type II$_1$ factors (see also \cite{DykemaSZ20115, DykemaSZ2017-unbounded} for some generalizations). 

The proof in \cite{HaagerupLarsen2000} uses Haagerup's work \cite{Haagerup-R-S-new} on a different proof for the multiplicative property of Voiculescu's $S$-transforms and the fact that any $R$-diagonal operator $T$ has the same distribution of a product $ax$ of free self-adjoint symmetric distributed elements. The proof in \cite{HaagerupSchultz2007} relies on a result of Nica--Speicher \cite{Nica-Speicher-1998duke} which says that for two $*$-free $R$-diagonal operators $S$ and $T$, the symmetrization of the distribution of $|S+T|$ is the free additive convolution of symmetrization distibutions of $|S|$ and $|T|$ (see Proposition \ref{Rdiag-additive}), and the following observation. Let $U$ be a Haar unitary operator $*$-free from $T$. For any $\lambda\in\mathbb{C}$, then
$|T-\lambda \unit|$ has the same distribution as $|T+|\lambda| U|$. Hence, for any $\lambda\in\mathbb{C}\backslash\{0\}$ fixed, the distribution of $|T-\lambda \unit|$ can be calculated as an addition of two $R$-diagonal operators. The problem is then reduced to understand the free additive convolution of a symmetric probability measure and a Bernoulli distribution. 
The above observations are the key to find a formula for the Fuglede--Kadison determinant \cite{FugledeKadison1952} of $T-\lambda\unit$.
The Brown measure of $T$ is then obtained by taking the distributional Laplacian of the logarithm of the Fuglede--Kadison determinant. In a joint work with Bercovici \cite{BercoviciZhong2022Rdiag}, we extend some ideas in this paper to study the Brown measure of a sum of two free random variables, one of which is $R$-diagonal.

We observe that subordination functions can be used to make Haagerup--Schultz's approach more transparent and several technical arguments in their work can be simplified when it cooperates with subordination ideas. The reformulation of Fugulede--Kadison formulas in terms of subordination functions allows us to simplify the original proof for the Brown measure formula of $R$-diagonal operators. By using a well-known result from logarithmic potential theory and rotation invariant property of the Brown measure of $R$-diagonal operators, we obtain a new proof of the Brown measure formula. We can also express the Brown measure formula in terms of subordination functions. 
  
The Hermitian reduction method is a very general and powerful technique to study the distributions of non-normal operators in free probability. It is known \cite{NSS2001-Rdiag} that an operator $T$ is $R$-diagonal if and only if the matrix-valued operator \begin{tiny}$\begin{bmatrix}0&T\\
	T^*&0\end{bmatrix}$\end{tiny} and $M_2(\mathbb C)$ are free with amalgamation over the algebra of diagonal matrices with a canonical conditional expectation. In \cite{BSS2018}, Belinschi--Śniady--Speicher used this characterization of $R$-diagonal operators and Hermitian reduction method to give a new proof of the Brown measure formula for $R$-diagonal operators. The subordination ideas found in \cite{HaagerupSchultz2007} allows us to find the connection between the two method which seems to be quite different in a first impression. 

The paper is organized as follows. In Section 2, we collect some preliminaries about Brown measure of $R$-diagonal operators to set up notations. In Section 3 and 4, we reformulate Haagerup--Schultz's work using subordination ideas. We present a new proof in Section 5 and discuss the connection with the Hermitian reduction method in Section 6. Finally, we show some applications in Section 7. 

\section{Preliminaries and notations}
Let $\CM$ be a finite von Neumann algebra with a faithful, normal, tracial
state $\tau$. We let $\tilde\CM$ denote the set of closed, densely defined
operators affiliated with $\CM$. In 1983, L.G.Brown \cite{Brown1986} introduced a spectral distribution for non-normal operators in $\tilde\CM$ which is a natural extension of eigenvalue counting measure of square matrices.  The definition of Brown measure can be extended to a larger class $\CMD$ of possibly unbounded operators in $\CM$ (\cite{HaagerupSchultz2007} and \cite[Appendix]{Brown1986}).

Given $T\in \tilde\CM$, and let $\mu_{|T|}$ be the spectral distribution of $|T|=(T^*T)^{1/2}$ with respect to $\tau$. Let $\CMD$ be the set of operators $T\in\tilde\CM$ satisfying the condition
\[
    \tau(\log^+|T|)= \int_0^\infty \log^+(t)\,\d\mu_{|T|}(t)<\infty.
\]
The \emph{Fuglede--Kadison determinant} of $T$ is defined by
\[
\Delta(T)= \exp\Big(\int_0^\infty \log t\,\d\mu_{|T|}(t)\Big).
\]
It is known \cite{HaagerupSchultz2007} that the set $\CMD$ is a subspace of $\tilde\CM$. In particular, for $T\in \CMD$ and $\lambda\in\mathbb{C}$, $T-\lambda \unit\in \CMD$. 
If $T\in  \CMD$ and $\Delta(T)>0$, then $ker(T)=\{0\}$, its inverse  $T^{-1}\in \CMD$, and $\Delta(T^{-1})=\frac{1}{\Delta(T)}$.

For $T\in\CMD$, the function $\lambda\mapsto f(\lambda):= \log \Delta(T-\lambda \unit)$ taking values in $ [-\infty, \infty)$ is subharmonic on $\mathbb{C}$. The \emph{Riesz measure} associated with this subharmonic function is called the \emph{Brown measure} of $T$ and we denote it by $\mu_T$.  More precisely, 
\[
   \d\mu_T= \frac{1}{2\pi}\nabla^2 f \,\d\lambda
\]
taken in the distribution sense. 
The Brown measure $\mu_T$ is the unique probability measure on $\mathbb{C}$
  satisfying
\begin{itemize}
    \item[(i)]
        \[
        \int_\C\log^+|z|\,\d\mu_T(z)<\infty,
        \]
    \item[(ii)] 
       \begin{equation*}
        \forall \lambda\in\C:\quad \log \Delta(T-\lambda \unit)= \int_\C \log|\lambda-z|\,\d \mu_T(z).
      \end{equation*}
  \end{itemize}

The family of $R$-diagonal operators was introduced by Nica--Speicher \cite{NicaSpeicher-Rdiag} which covers a large class of interesting operators in free probability theory. They have a number of remarkable symmetric properties.  
In \cite{HaagerupSchultz2007}, the class of $R$-diagonal operators was extended to unbounded case.
Recall that \cite{NicaSpeicher-Rdiag, HaagerupSchultz2007} $T\in\tilde\CM$ is $R$-diagonal if there exist a von Neumann algebra $\CN$, with a faithful, normal, tracial state, and $*$-free elements $U, H\in \CN$ with $U$ Haar unitary, $H\geq{0}$, and such that $T$ has the same $*$-distribution as $UH$. 
Let $V|T|$ be the polar decomposition of $T$.
If $T\in\tilde\CM$ is $R$-diagonal with ${\rm ker}(T)=0$, then $V$ is a Haar unitary which is $*$-free from $|T|$. 
	If $T\in \tilde\CM$ is $R$-diagonal with
	${\rm ker}(T)=0$, then $T$ has an
	inverse $T^{-1}\in\tilde\CM$, and $T^{-1}$ is $R$-diagonal as
	well.

Observe that the Brown measure of a $R$-diagonal operator $T$ is invariant under the rotation. Hence, the Brown measure of $T$ is completely determined by the distribution of $|T|$. In fact, 
the Brown measure of a $R$-diagonal operator $T$ can be expressed in terms of some analytic transformation of the spectral probability measure $\mu_{T^*T}$. 
For a probability measure $\mu$ on $\mathbb{R}_+=[0,\infty)$, define $\psi_\mu:\mathbb{C}\backslash (0,\infty)\rightarrow\mathbb{C}$ by
\[
    \psi_\mu(z)=\int_0^\infty \frac{zt}{1-zt}\,\d\mu(t), \qquad z\in
  \C\setminus (0,\infty).
\]
One can check that: (i) $\psi_\mu'(x)>0, x\in (-\infty, 0)$; (ii) $\psi_\mu(x)\rightarrow \delta -1$ as $x\rightarrow -\infty$ where $\delta=\mu(\{ 0\})$; (iii) $\psi_\mu(x) \rightarrow 0$ as $x\rightarrow 0$ (see \cite{BercoviciVoiculescu1993}). Hence, in a neighborhood $\CV_\mu$ of $(\delta-1,0)$, the inverse $\chi_\mu$ of $\psi_\mu$ exists: $\chi_\mu(z)=\psi_\mu^{\langle-1\rangle}(z), z\in \CV_\mu$. Then Voiculescu's $S$-transform is defined by
\[
   S_\mu(z)=\frac{z+1}{z}\chi_\mu(z), \quad z\in \CV_\mu.
\] 
For two probability measures $\mu,\nu$ on $\mathbb{R}_+$,
the $S$-transform has the remarkable property that $S_{\mu\boxtimes\nu}(z)=S_{\mu}(z)S_\nu(z)$ in a region where all these functions are defined.

\begin{notation}\label{notation:lambda-mu-1-2} 
Let $\mu$ be a probability measure on $[0,\infty)$ which
  is not a Dirac measure, and
  put
  \[
  \lambda_1(\mu) = \Bigg(\int_0^\infty
  \frac{1}{u^2}\,\d\mu(u)\Bigg)^{-\frac12} \quad {\rm and} \quad \lambda_2(\mu) = \Bigg(\int_0^\infty
  u^2\,\d\mu(u)\Bigg)^{\frac12},
  \]
  with the convention that $\infty^{-\frac12}=0$. 
Then $0\leq \lambda_1(\mu)<\lambda_2(\mu)\leq \infty$. 

Note that
  \[
  \lambda_2(\mu_{|T|})=\Bigg(\int_0^\infty u^2\,\d\mu_{|T|}(u)\Bigg)^\frac12
  = \|T\|_2
  \]
  and
  \[
  \lambda_1(\mu_{|T|})= \Bigg(\int_0^\infty u^{-2}\,\d\mu_{|T|}(u)\Bigg)^{-\frac12}
  = \|T^{-1}\|_2^{-1},
  \]
  where $ \|T^{-1}\|_2:=+\infty$ in case ${\rm ker}(T)\neq 0$. Often, we denote
  \[
    \lambda_1(T)= \lambda_2(\mu_{|T|}), \qquad\lambda_2(T)=\lambda_1(\mu_{|T|}).
  \]
\end{notation}

\vspace{.2cm}

The Brown measure of bounded $R$-diagonal operators was calculated by Haagerup-Larsen \cite{HaagerupLarsen2000} and the result was further extended by Haaagerup-Schultz \cite{HaagerupSchultz2007} for all $R$-diagonal operators in $\CMD$. 

\begin{theorem}
\label{thm-main}
  Let $T$ be an $R$--diagonal element in $\CMD$ with Brown measure
  $\mu_T$, and suppose $\mu_{T}$ is not a Dirac measure. The Brown measure $\mu_{T}$ is the unique rotationally invariant probability measure such that
  \begin{align*}
    \mu_{T}\{ \lambda\in\mathbb{C}: |\lambda|\leq r\}=
      \begin{cases}
         0, \quad & \text{for} \quad  r<\lambda_1(\mu_{|T|});\\
         1+S^{\langle -1 \rangle}_{\mu_{T^*T}}(r^{-2}),\quad &\text{for} \quad \lambda_1(\mu_{|T|})<r<\lambda_2(\mu_{|T|});\\
         1,\quad &\text{for}\quad r\geq \lambda_2(\mu_{|T|}).
      \end{cases}
  \end{align*}
When $\mu_{|T|}(\{0\})=\delta>0$, the Brown measure $\mu_T$ has a mass at the origin with the mass $\mu_T(\{0\})=\delta$. 
\end{theorem}

\begin{proposition}\label{prop:s-trans-monotonicity}
	\cite[Lemma 4]{Haagerup-Moller-LLN}
	Let $\mu$ be a probability measure on $[0,\infty)$ with $\delta=\mu(\{0\})<1$. Then $S_\mu$ is decreasing on $(\delta-1,0)$ and $S_\mu((\delta-1,0))=(b^{-1},a^{-1})$, where
	\[
	a=\left(\int_0^\infty \frac{1}{u}d\mu(u)\right)^{-1}, \quad b=\int_0^\infty u \,d\mu(u).
	\]
\end{proposition}
	Hence, the map $r\mapsto S_{\mu_{T^*T}}^{\langle -1 \rangle }(r^{-2})$ is inreasing on $(\lambda_1(\mu_{|T|}), \lambda_2(\mu_{|T|}))$ with range $(\delta-1,0)$, where $\delta=\mu(\{0\})$.

Given a probability measure $\sigma$ on $\mathbb{R}$, let $\tilde\sigma$ denote the
{\it symmetrization} of $\mu$. That is, $\tilde\sigma$ is a symmetric probability measure given
by
\[
\tilde\sigma(B) = \textstyle{\frac12}(\sigma(B)+\sigma(-B)), 
\]
where $B$ is any Borel measurable set in $\mathbb{R}$.
 Let $\mu, \nu$ be probability measures on $\mathbb{R}$, their free convolution is denoted by $\mu\boxplus \nu$.	
The following result is a powerful tool to calculate the distribution of the sum of $R$-diagonal operators. 
\begin{proposition} (\cite{HaagerupSchultz2007, Nica-Speicher-1998duke})
	\label{Rdiag-additive} 
	Let $S,T\in\tilde\CM$ be $\ast$--free $R$--diagonal
	elements. Then $S+T\in\tilde\CM$ is $R$-diagonal and 
	\begin{equation}
	\tilde\mu_{|S+T|}= \tilde\mu_{|S|}\boxplus \tilde\mu_{|T|}.
	\end{equation}
\end{proposition}	

The calculation of free additive convolution of two probability measures $\mu,\nu$ on $\mathbb{R}$ uses Voiculescu's $R$-transform. 
For a probability measure $\mu$ on $\mathbb{R}$, its Cauchy transform is defined as
\[
G_\mu(z)=\int_\R \frac{1}{z-x}d\mu(x), \qquad z\in\mathbb{C}^+;
\]
and the $F$-transform is its reciprocal, $F_\mu(z)=1/G_\mu(z)$. The Voiculescu's $R$-transform is defined by $R_\mu(z)=G_\mu^{\langle -1\rangle}(z)-1/z$ for $z$ in a truncated Stolz angle $\{z\in\C\ : \Im z>\beta, \left\vert \Re z\right\vert<\alpha (\Im z)\}$ for some $\alpha, \beta>0$. 
The $R$-transform has remarkable property that $R_{\mu\boxplus\nu}(z)=R_\mu(z)+R_\nu(z)$ for $z$ in some domain where all these functions are defined. Another way to study free convolution is to use subordination functions, which are analytic functions $\omega_1, \omega_2:\mathbb{C}^+\rightarrow\mathbb{C}^+$ such that
\[
F_{\mu}(\omega_1(z))=F_{\nu}(\omega_2(z))=F_{\mu\boxplus\nu}(z), \qquad z\in\mathbb{C}^+. 
\]
 For a probability measure $\mu$ on $\mathbb{R}$, denote $H_\mu(z)=F_\mu(z)-z$. In \cite{BB2007new}, Belinschi and Bercovici showed that $\omega_1, \omega_2$ can be obtained from the following fixed point equations
\[
\omega_1(z)=z+H_{\mu_2}(z+H_{\mu_1}(\omega_1(z))), \qquad
\omega_2(z)=z+H_{\mu_1}(z+H_{\mu_2}(\omega_2(z))).
\]
In operator-valued free probability theory, all the above analytic transformations have operator-valued analogues. See \cite[Chapter 9]{MingoSpeicherBook} and \cite{DVV-general} for example.

\section{The subordinations and Haagerup--Schultz's approach}
\label{section:sub-HS}
Let $T\in\tilde\CM$ be a $R$-diagonal operator. The goal of this section is to understand the distribution of $T-\lambda \unit$ and some technical functions introduced in \cite{HaagerupSchultz2007}, but from the subordination function perspective. Let $U\in\CM$ be a Haar unitary which is $\ast$--free from $T$. 
Then $U^*$ is a Haar unitary which is also $\ast$--free from $T$. 
Consequently, the operator $U^*T$ is $R$-diagonal and has the same $\ast$-distribution as $T$ \cite{NicaSpeicher-Rdiag}. In addition, $-\lambda U$ has the same distribution as $|\lambda| U$.  Hence, for $\lambda\in \mathbb{C}$, we have the following identification of distributions
\begin{equation}\label{dist: T-lambda-same-dist}
  |T-\lambda \unit| \underset{\CD}{\sim} |U^*T-\lambda \unit| 
  \underset{\CD}{\sim} |T-\lambda U|\underset{\CD}{\sim} |T+|\lambda| U |. 
\end{equation}

Throughout this section,   $\mu_{|T|}$ denotes the spectral measure of the positive operator $|T|$ and $\mu=\tilde{\mu}_{|T|}$ is its symmetrization, if we do not declare assumption separately. 
By Proposition \ref{Rdiag-additive} and \eqref{dist: T-lambda-same-dist}, it follows that
\[
\tilde{\mu}_{|T-\lambda\unit|} = \tilde{\mu}_{|T|}\boxplus
\tilde{\mu}_{|\lambda|\unit}= \mu\boxplus \nu,
\]
where $\nu= \frac12 (\delta_{-|\lambda|}+\delta_{|\lambda|})$. We then observe that for any operator $x\in \mathcal{M}$ and $\varepsilon>0$, we have
\[
   \frac{1}{2}\frac{d}{d{t}}\big(\log\Delta(x^*x+t^2\unit)\big)=t\tau((x^*x+t^2\unit)^{-1}).
\]
The above equation is connected with Cauchy transform of $\tilde{\mu}_{|x|}$ by
\begin{equation}
\label{eqn:G_tilde_general}
   G_{\tilde{\mu}_{|x|}}(it)=\int_\mathbb{R} \frac{1}{it-u}d\tilde{\mu}_{|x|} (u)=\int_\mathbb{R}\frac{- it}{t^2+u^2}d\tilde{\mu}_{|x|} (u)=-it\tau((x^*x+t^2\unit)^{-1}),
\end{equation}
where we used the symmetric property of $\tilde{\mu}_{|x|}$. 

\begin{definition}\label{defn-h} Let $T\in\tilde\CM$ be an $R$--diagonal operator, and
	define
	\[
	h(s)= s\,\tau\big((T\cc T+s^2\unit)^{-1}\big), \qquad s>0.
	\]
	Moreover, for $\lambda\in\C\setminus\{0\}$, set
	\[
	h_\lambda(t) = t\,\tau\big([(T-\lambda\unit)\cc
	(T-\lambda\unit)+t^2\unit]^{-1}\big), \qquad t>0.
	\]
\end{definition}

Let $\omega_1, \omega_2$ be subordination functions such that $ F_{\tilde{\mu}_{|T-\lambda\unit|}}(z)= F_{\tilde{\mu}_{|T|}}(\omega_1(z))=
F_{\nu}(\omega_2(z))$. That is
\begin{equation}\label{eqn-subordination-F}
  F_{\mu\boxplus\nu}(z)=F_\mu(\omega_1(z))=F_\nu(\omega_2(z)).
\end{equation}
Note that $\omega_1, \omega_2$ are functions of $z$ that depend on $T$ and $\vert\lambda\vert$. 
As noted above, by symmetrization, the Cauchy transforms satisfy 
\begin{equation}\label{eqn:G_tilde_h_lambda}
  G_{\tilde{\mu}_{|T-\lambda\unit|}}(it)=-ih_\lambda(t), \qquad   G_{\tilde{\mu}_{|T|}}(is)=-ih(s)
\end{equation}
and the derivative of logarithm of Fuglede--Kadison determinants can be expressed as
\begin{equation}
 \label{eqn:derivative-log-det}
	\begin{aligned}
	 \frac{1}{2}\frac{d}{d{t}}\big(\log\Delta((T-\lambda\unit)\cc
	 (T-\lambda\unit)+t^2\unit)\big)&=h_\lambda(t);\\
	\frac{1}{2}\frac{d}{d{s}}\big(\log\Delta(T\cc T+s^2\unit)\big)&=h(s).
	\end{aligned}
\end{equation}

The subordination function builds a connection between $h_\lambda(t)$ and $h(s)$. 
\begin{proposition}
 \label{prop:3.2-def-W}
	For $\lambda\in \mathbb{C}\backslash\{0\}$ and $t>0$, the complex numbers $\omega_1(it), \omega_2(it)$ are pure imaginary. Set 
	$\omega_1(it)=iW(t)$ for some $W(t)>t$.
	We have 
	\begin{equation}
	  \label{eqn:sub-scalar-lambda-W}
	h_\lambda(t)=h(-i\omega_1(it))=h(W(t)).
	\end{equation}
\end{proposition}
\begin{proof}
	In a domain of the form $\Gamma_{\alpha,\beta}=\{ z=x+iy: y>\beta, |x/y|<\alpha \}$ for some $\alpha, \beta>0$ (see \cite{BercoviciVoiculescu1993}), we have 
	\[
	\omega_1(z)=F^{\langle -1\rangle}_{\tilde{\mu}_{|T|}}\left (F_{\tilde{\mu}_{|T-\lambda\unit|}}(z) \right), 
	\qquad z\in\Gamma_{\alpha,\beta}.
	\]
	Hence, by \eqref{eqn:G_tilde_h_lambda}, $\omega_1(it)$ is purely imaginary for $t>\beta$ and therefore $\omega_1(it)$ is purely imaginary for all $t>0$ by analytic extension. The same argument works for $\omega_2$. 
	 Note that we can express subordination function as $\omega_1(z)=F_\sigma(z)$ for some probability measure $\sigma$ on $\mathbb{R}$ (see \cite[Definition 2.1]{Len2007jfa} for example). Hence, by the asymptotic of $F$-transform (see \cite[Proposition 5.2]{BercoviciVoiculescu1993} for example), 
	we have $\Im\omega_1(z)>\Im (z)$.
	Set $W(t)=-i\omega_1(it)$. We have $W(t)=\Im \omega_1(it)>\Im (it)=t$. Finally, 
	\[
	h_\lambda(t)=-\Im G_{\tilde{\mu}_{|T-\lambda\unit|}}(it)
	=-\Im G_{\tilde{\mu}_{|T|}} (\omega_1(it))=h(W(t)),
	\]
	which is exactly the subordination relation \eqref{eqn:sub-scalar-lambda-W}. 
\end{proof}
\begin{definition}\label{defn:s-lambda-t}
	For $\lambda\in \mathbb{C}\backslash\{0\}$ and $t>0$, set 
	\[
	s(|\lambda|,t)=W(t)=\Im\omega_1(it). 
	\]
\end{definition}

We now outline the calculation of the Brown measure for $T$ as four steps:
\begin{enumerate}[(1)]
	\item The subordination function $s(|\lambda|,t)$ depends on $\lambda$. Study the analytic property of $\lim_{t\rightarrow 0} s(|\lambda|,t)$ as a function of $\lambda$. This is equivalent to studying the boundary value of the subordination function $\omega_1(it)$ as $t$ goes to zero. 
		\item Express $\Delta(|T-\lambda\unit|^2+t^2\unit)$ in terms of $\Delta(|T|^2+s^2\unit)$ where $s$ is some subordination function. To do so, we need to evaluate the following integrals
	\[
	\int_{t_0}^t h_\lambda(t)dt, \qquad\text{and} \qquad\int_{s_0}^s h(s)ds.
	\]
	The key in this step is to find a tractable formula for subordination functions. 
	\item  Find a formula for $\Delta(T-\lambda\unit)^2=\lim_{t\rightarrow 0}\Delta(|T-\lambda\unit|^2+t^2)$.
	\item Finally, calculate the Brown measure of $T$. 
\end{enumerate}

Following \cite{HaagerupSchultz2007}, 
recall $\mu=\tilde{\mu}_{|T|}$, the function $h$ in Definition \ref{defn-h} reads 
\begin{equation}\label{defn:h-s}
h(s)=\int_\mathbb{R} \frac{s}{s^2+u^2}d\mu(u),\quad s>0. 
\end{equation}
Note that the above definition is the same as the function in \cite[Lemma 4.6]{HaagerupSchultz2007} because 
\[
  \int_\mathbb{R} \frac{s}{s^2+u^2}d\mu(u)=\int_0^\infty \frac{s}{s^2+u^2}d\mu_{|T|}(u).
\]
 Then put
  \begin{equation}
    k(s,t)=(s-t)\bigg(\frac{1}{h(s)}-s+t\bigg), \qquad s>0,\, t\in\R.
  \end{equation}
In particular, we have
\[
   k(s,0)=s\bigg(\frac{1}{h(s)}-s \bigg)=\frac{n(s)}{d(s)}
\]
where 
\[
  n(s)=\int_\mathbb{R}\frac{u^2}{s^2+u^2}d\mu(u) \qquad\text{and}\qquad
   d(s)=\int_\mathbb{R}\frac{1}{s^2+u^2}d\mu(u).
\]

 Recall that $\lambda_1(\mu), \lambda_2(\mu)$ were defined in Notation \ref{notation:lambda-mu-1-2}. 
\begin{lemma}\cite[Lemma 4.8]{HaagerupSchultz2007}
	\label{lemma:monotonicity-k-s-0}
The map $s\mapsto k(s,0)$ is a strictly increasing bijection of $(0,\infty)$
onto $(\lambda_1(\mu)^2,\lambda_2(\mu)^2)$. 
\end{lemma}  
\begin{proof}
We have $k(s,0)=\frac{s}{h(s)}-s^2$. One can verify that
\[
   \frac{d}{ds}\frac{s}{h(s)}=\frac{d}{ds}\frac{1}{\int_\mathbb{R}\frac{1}{s^2+u^2}d\mu(u)}
     =2s\cdot\frac{\int_\mathbb{R}\frac{1}{(s^2+u^2)^2}d\mu(u)}{\left(\int_\mathbb{R}\frac{1}{s^2+u^2}d\mu(u)\right)^2}.
\]
Since $\mu$ is non-trivial, the second factor in the right hand side of the above equation is strictly bigger than one by Cauchy-Schwartz inequality. Hence, $\frac{d}{ds}k(s,0)>0$. We can verify directly
\[
  \lim_{s\rightarrow 0^+} k(s,0)=\frac{1}{\int_\mathbb{R}\frac{1}{u^2}d\mu(u)}
  \qquad\text{and}\qquad \lim_{s\rightarrow \infty} k(s,0)=\int_\mathbb{R}u^2 d\mu(u).
\]
The result then follows. 
\end{proof}

    \begin{lemma}\label{lemma:monotoniciity-k-s-t}
    \cite[Lemma 4.8]{HaagerupSchultz2007}
  The function $k$ is an analytic on $(0,\infty)\times\R$. Moreover,
  for $t>0$ the map $s\mapsto k(s,t)$ is a strictly increasing
  bijection of $(t,\infty)$ onto $(0,\infty)$
  \end{lemma}
\begin{proof}
The function $k(s,t)$ can be expressed as
\[
  k(s,t)=\bigg( 1-\frac{t}{s} \bigg)s\bigg(\frac{1}{h(s)}-s \bigg)+(s-t)t=k(s,0)\bigg( 1-\frac{t}{s} \bigg)+(s-t)t.
\]
It follows from Lemma \ref{lemma:monotonicity-k-s-0} that $s\mapsto k(s,t)$ is strictly increasing for $s>t$. One can check directly that $k(s,t)\rightarrow \infty$ as $s\rightarrow \infty$ for any $t>0$ fixed. 
\end{proof}
The following result shows that our definition for $s(|\lambda|,t)=W(t)$ coincides with the function $s(|\lambda|,t)$ in \cite[Definition 4.9]{HaagerupSchultz2007}.   In \cite{HaagerupSchultz2007}, Haagerup--Schultz studied various useful properties of the function $s$. 
\begin{proposition}\label{prop:equivalence-two-eqns}
The analytic function $\omega_1$ is the unique solution of the fixed point equation for $\omega\in\mathbb{C}^+$:
\begin{equation}\label{eqn:fixed-point}
  \omega(z)=z+H_2(z+H_1(\omega(z))),  \qquad z\in\mathbb{C}^+
\end{equation}
where 
\[
  H_1(z)=F_{{\mu}}(z)-z, \quad \text{and} \quad H_2(z)=F_{\nu}(z)-z.
\]
For $r>0$, the function $s(r, t)=W(t)=\Im \omega_1(it)$ is the unique solution $s\in (t,\infty)$ to the equation $k(s,t)=r^2$, which is
\begin{equation}\label{eqn:fixed-point-s-t}
  (s-t)^2-\frac{s-t}{h(s)}+r^2=0.
\end{equation}
\end{proposition}
\begin{proof}
It is known \cite{BB2007new} that the subordination function \eqref{eqn-subordination-F} satisfies the  fixed point equation \eqref{eqn:fixed-point}. 
We next show that $s(r,t)$ is the unique solution to the equation \eqref{eqn:fixed-point-s-t}.
Given $\lambda\in\mathbb{C}\backslash\{0\}$, recall that $\nu=\frac{1}{2}(\delta_{|\lambda|}+\delta_{-|\lambda|})$.
We calculate that 
$ G_\nu(z)=\frac{1}{2}\left( \frac{1}{z-|\lambda|}+\frac{1}{z+|\lambda|} \right)$ and hence
$H_2(z)=F_\nu(z)-z=\frac{-|\lambda|^2}{z}.$
The fixed point equation \eqref{eqn:fixed-point} is expressed as
\[
  \omega(z)=z-\frac{|\lambda|^2}{z+H_1(\omega(z))}
\]
which is equivalent to 
\begin{equation}\label{eqn:fixed-point-omega-z}
- (\omega(z)-z)^2+\frac{\omega(z)-z}{G_{{\mu}}(\omega(z))}+|\lambda|^2=0.
\end{equation}
When $z=it$, $\omega_1(it)=is(|\lambda|, t)$, using $F_{{\mu}}(is)=1/G_{{\mu}}(is)=i/h(s)$, the above equation is further reduced to \eqref{eqn:fixed-point-s-t}
\[
(s-t)^2-\frac{s-t}{h(s)}+|\lambda|^2=0.
\]
The uniqueness of the fixed point equation \eqref{eqn:fixed-point} implies that above equation has a unique solution. The uniqueness also follows from 
Lemma \ref{lemma:monotoniciity-k-s-t}.
\end{proof}

\begin{remark}
For $\lambda\in\mathbb{C}\backslash\{0\}$, and set
\[
  K(\omega, z)=(\omega-z)\left( \omega-z-\frac{1}{G_{{\mu}}(\omega)} \right),
  \quad  z, \omega\in\mathbb{C}^+.
\]
For any $z\in\mathbb{C}^+$, the equivalence between \eqref{eqn:fixed-point-omega-z} and the fixed point equation \eqref{eqn:fixed-point} as shown in the proof of Proposition \ref{prop:equivalence-two-eqns} implies that there is a unique solution $\omega\in \mathbb{C}^+$ such that $K(\omega,z)=|\lambda|^2$. If $z=it$, then the solution $\omega=i s$ for some $s>t$. Moreover, $K(is, it)=k(s,t)$.
\end{remark}

  
  \begin{definition}\label{def-for-s-t} 
  For $r\in (\lambda_1(\mu), \lambda_2(\mu))$, let $s_0$ denote the unique solution $s\in (0,\infty)$ to the equation $k(s,0)=r^2$ as in Lemma \ref{lemma:monotoniciity-k-s-t}.
For $r>0$, define $s(r,0)$ as follows.
\begin{equation}
\label{defn:s-lambda-0}
s(r, 0)=
  \begin{cases}
   0, \quad \text{if} \quad 0<r\leq \lambda_1(\mu);\\
       s_0, \quad\text{if}\quad r\in (\lambda_1(\mu),\lambda_2(\mu));\\
       +\infty,\quad \text{if}, \quad r \geq \lambda_2(\mu).
  \end{cases}
\end{equation}
  \end{definition}
  
\begin{proposition}\label{prop:s(r,0)-increasing}
The function $r\mapsto s(r,0)$ is increasing on $[0,\infty)$. It is strictly increasing on $(\lambda_1(\mu),\lambda_2(\mu))$ and a bijection of $(\lambda_1(\mu),\lambda_2(\mu))$ onto $(0,\infty)$.
\end{proposition}
\begin{proof}
It follows from the fact that the map $s\mapsto k(s,0)$ is a strictly increasing bijection of $(0,\infty)$ onto $(\lambda_1(\mu)^2,\lambda_2(\mu)^2))$ and the definition that $s(r,0)$ is the unique solution to the equation $k(s,0)=r^2$ when $\lambda_1(\mu)<r<\lambda_2(\mu)$.
\end{proof}

It is known \cite{Belinschi2008} that the subordination functions $\omega_1, \omega_2$ extend continuously to $\mathbb{R}$ with values in the extended complex plane $\overline{\mathbb{C}}$.  The following result clarifies that the extension of $\omega_1, \omega_2$ at $0$ is a nice function of $\lambda$. 
\begin{lemma}\label{lemma-limit-s-0} 
The function $(r,t)\mapsto s(r,t)$ is analytic in $(0,\infty)\times (0,\infty)$. 
We have $s(r, t)> t$. 
Moreover, 
\begin{equation}\label{eqn:limit_omega1}
\lim_{t\rightarrow 0^+}s(r, t)=s(r,0),  
\quad r>0.
\end{equation}
In addition,
we have $\lim_{t\rightarrow\infty}(s(r,t)-t)=0$ for all $r>0$.

When $0<r\leq \lambda_1(\mu)$, we have 
\[
   \lim_{t\rightarrow 0^+} \frac{t}{s(r,t)}=\frac{\lambda_1(\mu)^2-r^2}{\lambda_1(\mu)^2}.
\]
When $r\geq \lambda_2(\mu)$, we have 
\[
  \lim_{t\rightarrow 0^+}s(r,t)t=r^2-\lambda_2(\mu)^2.
 \]
\end{lemma}

\begin{proof}
The analyticity of the function 
$(r,t)\mapsto s(r,t)$ was proved in \cite[Lemma 4.10]{HaagerupSchultz2007}.
The limit $\lim_{t\rightarrow 0^+}s(r,t)$ was proved in \cite[Lemma 4.10 and Remark 4.11]{HaagerupSchultz2007} and this is equivalent to \eqref{eqn:limit_omega1}. 

We now show that $\lim_{t\rightarrow\infty}(s(r,t)-t)=0$. As we mentioned in the proof of Proposition \ref{prop:3.2-def-W}, we can express subordination function as $\omega_1(z)=F_\sigma(z)$ for some probability measure $\sigma$ on $\mathbb{R}$. Hence, by the asymptotic of $F$-transform (see \cite[Proposition 5.2]{BercoviciVoiculescu1993} for example), we have
\[
 \lim_{t\rightarrow\infty}(s(r,t)-t)=\lim_{t\rightarrow} F_\sigma(it)-it=0.
\]

By the monotone convergence theorem, 
\[
  \lim_{s\rightarrow 0^+}\frac{h(s)}{s}=\lim_{s\rightarrow 0^+}\int_\mathbb{R} \frac{1}{s^2+u^2}d\mu(u)=\frac{1}{\lambda_1(\mu)^2},
\]
\[
  \lim_{s\rightarrow \infty}sh(s)=\lim_{s\rightarrow\infty}\int_0^\infty \frac{s^2}{s^2+u^2}d\mu(u)=1.
  \]
  \[
  \lim_{s\rightarrow 0^+} s^2\Big( 1-sh(s) \Big)=
      \lim_{s\rightarrow 0^+} \int_0^\infty\frac{s^2 u^2}{s^2+u^2}d\mu(u)=\lambda_2(\mu)^2.
\]
Hence, 
when $0<r\leq \lambda_1(\mu)$, by \eqref{eqn:fixed-point-s-t} and \eqref{eqn:limit_omega1}, 
\begin{align*}
r^2&=\lim_{t\rightarrow 0^+}\left( -(s(r,t)-t)^2+\frac{s(r,t)-t}{h(s(r,t))}\right)\\
  &=\lim_{s\rightarrow 0^+}\frac{s}{h(s)}\lim_{t\rightarrow 0^+}\left(1-\frac{t}{s(r,t)}\right).
\end{align*}
Therefore, for $0<r\leq \lambda_1(\mu)$,
\[
   \lim_{t\rightarrow 0^+} \frac{t}{s(r,t)}=\frac{\lambda_1(\mu)^2-r^2}{\lambda_1(\mu)^2}.
\]

For the case  $r\geq \lambda_2(\mu)$, we rewrite \eqref{eqn:fixed-point-s-t} as
\begin{align*}
r^2&=-(s-t)^2+\frac{s-t}{h(s)}\\
      &=\frac{s^2\Big( 1-sh(s) \Big) }{sh(s)}
         +st\Big( 2sh(s)-1\Big)-t^2 sh(s).
\end{align*}
Hence, taking the limit of the above identity, when $r\geq \lambda_2(\mu)$, by \eqref{eqn:limit_omega1}, 
\begin{align*}
  \lim_{t\rightarrow 0^+}s(r,t)t&=r^2-\lim_{s\rightarrow \infty}\frac{s^2\Big( 1-sh(s) \Big) }{sh(s)}\\
    &=r^2-\lambda_2(\mu)^2.
\end{align*}
This finishes the proof. 
\end{proof}

We summarize the discussions above. 
Let $T\in\tilde\CM$ be an $R$--diagonal element, let
  $\lambda\in\C\setminus\{0\}$. Let $\mu= \tilde\mu_{|T|}$, and let
  $s(|\lambda|,t)$ be as in Definition~\ref{def-for-s-t}. Then
  \begin{equation}
   \label{eqn:identity-h-lambda-h}
   h_\lambda(t)=h(s(|\lambda|,t)), \qquad t>0.
  \end{equation}
 Set $s=s(|\lambda|,t)$, then $s$ satisfies
 \begin{equation}\label{eqn:fixed-point-s-t-abs-lambda}
  (s-t)^2-\frac{s-t}{h(s)}+|\lambda|^2=0.
 \end{equation}
Hence 
\[
  1-4|\lambda|^2 h(s)^2=\Big( 1-2h(s)(s-t)\Big)^2\geq 0.
\]
 Moreover, by \cite[Lemma 4.12]{HaagerupSchultz2007},  there exists a $t_\lambda >0$ such that when
    $t>t_\lambda$ we have 
  \begin{equation}
  \label{eqn:s-t-for-large-t}
  	      s-t=\frac{1-\sqrt{1-4|\lambda|^2h(s)^2}}{2h(s)}.
  \end{equation}
Indeed, $\lim_{t\rightarrow\infty}(s-t)=0$ and $\frac{1}{h(s)}\rightarrow 0$ as $s\rightarrow\infty$. Solve the quadratic equation \eqref{eqn:fixed-point-s-t-abs-lambda} for $s-t$ to get the formula for $s-t$ when $t$ is large. However, for small $t$, it is not clear if we need to change the sign in the quadratic formula from the right hand side of \eqref{eqn:s-t-for-large-t}.

\begin{lemma}\label{lemma-limi-omega-2}
Let $\omega_2$ be  the subordination function such that $F_\nu(\omega_2(z))=F_{\tilde{\mu}_{|T-\lambda\unit|}}(z)$. Then, for any $t>0$, 
\[
   \omega_2(it)=i\frac{|\lambda|^2}{s(|\lambda|, t)-t}=\frac{|\lambda|^2}{it-i\omega_1(it)}.
\]
Moreover, 
\begin{align*}
\lim_{t\rightarrow 0^+}\omega_2(it)=
  \begin{cases}
     \infty, \quad  &\text{if} \quad 0<|\lambda|\leq \lambda_1(\mu);\\
       i\frac{|\lambda|^2}{s(|\lambda|,0)}, \quad &\text{if}\quad |\lambda|\in (\lambda_1(\mu),\lambda_2(\mu));\\
       0,\quad & \text{if}, \quad |\lambda|\geq \lambda_2(\mu).
  \end{cases}
\end{align*}
When $0<|\lambda|\leq \lambda_1(\mu)$, we have
\begin{align*}
  \lim_{t\rightarrow 0^+}-it\cdot\omega_2(it)
       &=\lambda_1(\mu)^2-|\lambda|^2.
\end{align*}
When $|\lambda|\geq \lambda_2(\mu)$, we have 
\begin{align*}
 \lim_{t\rightarrow 0^+}\frac{it}{\omega_2(it)}
      =\frac{|\lambda|^2-\lambda_2(\mu)}{|\lambda|^2}.
\end{align*}
\end{lemma}
\begin{proof}
We have $\omega_1(z)+\omega_2(z)=z+F_{\tilde{\mu}_{|T-\lambda\unit|}}(z)$. Recall that 
$G_{\tilde{\mu}_{|T-\lambda\unit|}}(it)=-ih_\lambda(t)$ and $G_{\tilde{\mu}_{|T-\lambda\unit|}}(it)=G_{\tilde{\mu}_{|T|}}(\omega_1(it))=G_{\tilde{\mu}_{|T|}}(i s(|\lambda|,t))=-ih(s(|\lambda|,t))$. Hence, 
\[
 \omega_1(it)+\omega_2(it)=it+F_{\tilde{\mu}_{|T-\lambda\unit|}}(it)
\]
is rewritten as
\[
  is(|\lambda|,t)+\omega_2(it)=it+\frac{1}{-ih(s(|\lambda|,t))}.
\]
On the other hand, $s=s(|\lambda|,t)$ satisfies the identity
\[
   (s-t)^2-\frac{s-t}{h(s)}+|\lambda|^2=0,
\]
which implies $h(s)=\frac{s-t}{(s-t)^2+|\lambda|^2}$. Therefore, 
\[
   \omega_2(it)=i\left[ t-s+\frac{(s-t)^2+|\lambda|^2}{s-t}\right]=i\frac{|\lambda|^2}{s(|\lambda|,t)-t}=\frac{|\lambda|^2}{it-\omega_1(it)}.
\]
By taking the limit as $t\rightarrow 0^+$ and using Lemma \ref{lemma-limit-s-0}, we obtain $\lim_{t\rightarrow 0^+}\omega_2(it)$.

When $0<|\lambda|\leq \lambda_1(\mu)$, applying Lemma \ref{lemma-limit-s-0} and above formulas, we have
\begin{align*}
  \lim_{t\rightarrow 0^+}-it\cdot\omega_2(it)
     &= \lim_{t\rightarrow 0^+}\frac{t|\lambda|^2}{s(|\lambda|,t)-t}\\
       &=\lim_{t\rightarrow 0^+}\frac{|\lambda|^2}{s(|\lambda|,t)/t-1}\\
       &=\lambda_1(\mu)^2-|\lambda|^2.
\end{align*}
When $|\lambda|\geq \lambda_2(\mu)$, we have 
\begin{align*}
 \lim_{t\rightarrow 0^+}\frac{it}{\omega_2(it)}
    &=\lim_{t\rightarrow 0^+}\frac{s(|\lambda|,t)t-t^2}{|\lambda|^2}
      =\frac{|\lambda|^2-\lambda_2(\mu)}{|\lambda|^2}.
\end{align*}
\end{proof}

We give a proof of the following lemma from \cite{HaagerupSchultz2007} for reader's convenience. 
\begin{lemma}\label{lemma-delta-star-s-t} 
\cite[Lemma 4.14]{HaagerupSchultz2007}
Let $T$ be an unbounded $R$--diagonal element in $\CMD$,
  let $\lambda\in\C\setminus\{0\}$, and let $t>0$. With
  $\mu=\mu_{|T|}$ and $s(|\lambda|,t)$ as in Definition~\ref{defn:s-lambda-t},
  we then have:
  \begin{equation}\label{eqn:delta-star-s-t}
    \Delta\big((T-\lambda\unit)\cc(T-\lambda\unit)+t^2\unit\big)=
    \frac{|\lambda|^2}{|\lambda|^2 + (s(|\lambda|,t)-t)^2} \Delta\big(T\cc
    T + s(|\lambda|,t)^2\unit\big).
  \end{equation}
\end{lemma}
\begin{proof}
We use the relation between function $h, h_\lambda$ and the logarithm of Fuglede--Kadison determinant \eqref{eqn:derivative-log-det}. 
The function $s(|\lambda|,t)$ satisfies \eqref{eqn:fixed-point-s-t}, which is
\[
    (s-t)^2-\frac{s-t}{h(s)}+|\lambda|^2=0. 
\]
Fix $\lambda\in\mathbb{C}\backslash\{0\}$, we set $W(t)=s(|\lambda|,t)$. 
Hence, from Equation \eqref{eqn:sub-scalar-lambda-W} and \eqref{eqn:fixed-point-s-t}, for any $v>0$, we have
\[
  h_\lambda(v)= (h\circ W)(v)=\frac{W(v)-v}{(W(v)-v)^2+|\lambda|^2}.
\]

Hence, for $0<t_0<t$, we have 
\begin{align*}
 &\frac{1}{2}\bigg(\log\frac{\Delta((T-\lambda\unit)\cc
 (T-\lambda\unit)+t^2)}{\Delta((T-\lambda\unit)\cc
 (T-\lambda\unit)+t_0^2)}\bigg)\\
 =& \int_{t_0}^t h_\lambda(v)dv\\
 =&\int_{t_0}^t h(W(v))W'(v)dv  -\int_{t_0}^t \frac{W(v)-v}{(W(v)-v)^2+|\lambda|^2}(W'(v)-1)dv \\
=&\int_{W(t_0)}^{W(t)} h(s)ds
  -\int_{W(t_0)-t_0}^{W(t)-t} \frac{x}{x^2+|\lambda|^2} dx\\
 =&\frac{1}{2}\log \frac{\Delta(T^*T +W(t)^2\unit) }{\Delta(T^*T+W(t_0)^2\unit)}
   -\frac{1}{2}\log \frac{ (W(t)-t)^2+|\lambda|^2 }{(W(t_0)-t_0)^2+|\lambda|^2}.
\end{align*}
Hence there is a constant $C$ such that
\begin{align*}
&\log\Delta((T-\lambda\unit)\cc
(T-\lambda\unit)+t^2\unit)\\
=&\log\Delta(T^*T+W(t)^2)-\log(|(W(t)-t)^2+|\lambda|^2)+C.
\end{align*}
Recall that $\lim_{t\rightarrow\infty}W(t)-t=\lim_{t\rightarrow\infty}(s(|\lambda|,t)-t)=0$, we then figure out that $C=\log|\lambda|^2$. The result then follows. 
\end{proof}

By Lemma \ref{lemma-delta-star-s-t} and the limit of $s(|\lambda|,t)$ as in \eqref{eqn:limit_omega1}, we then have the formula for Fuglede--Kadison determinant. 
\begin{theorem}\label{thm5.15}
\cite[Theorem 4.15]{HaagerupSchultz2007}
  Let $T\in\CMD$ be $R$--diagonal, let $\mu=\mu_{|T|}$, and let
  $s(|\lambda|,0)$ be as in Definition~\ref{def-for-s-t}.
  \begin{itemize}
    \item[(i)] If $\lambda_1(\mu)<|\lambda|<\lambda_2(\mu)$, then
      \begin{equation}
        \label{eqn:det-general-t-0}
        \Delta(T-\lambda\unit)=
        \Bigg(\frac{|\lambda|^2}{|\lambda|^2+s(|\lambda|,0)^2}\,
        \Delta(T\cc T +s(|\lambda|,0)^2\unit)\Bigg)^\frac12.
      \end{equation}
    \item[(ii)] If $|\lambda|\leq \lambda_1(\mu)$, then
    $\Delta(T-\lambda\unit)= \Delta(T)$.
    \item[(iii)] If $|\lambda|\geq \lambda_2(\mu)$, then
    $\Delta(T-\lambda\unit)=|\lambda|$.
  \end{itemize}
\end{theorem}

\vspace{.2cm}

\section{The Brown measure of $R$-diagonal operators}
\label{section:Brown-Rdiag-proof-1}

In this section, we calculate the gradient functions of the logarithmic Fuglede--Kadison determinant of $T-\lambda\unit$ and provide a proof for the Brown measure formula of $R$-diagonal operators which simplifies Haagerup--Schultz's original arguments.
\begin{lemma}\label{lemma-s-t-Cauchy}
Let $T\in\CMD$ be $R$--diagonal and $\lambda\in\mathbb{C}\backslash\{0\}$. Then
\[
  \tau\huge( (\lambda\unit-T)[(\lambda \unit-T)^*(\lambda \unit-T)+t^2\unit]^{-1} \huge)
     =\frac{\lambda}{|\lambda|^2}\frac{(s-t)^2}{|\lambda|^2+(s-t)^2}
\]
and
\[
  \tau\huge( (\lambda\unit-T)^*[(\lambda\unit-T)^*(\lambda\unit-T)+t^2\unit]^{-1} \huge)
     =\frac{\overline{\lambda}}{|\lambda|^2}\frac{(s-t)^2}{|\lambda|^2+(s-t)^2}
      \]
where $s=s(|\lambda|,t)$.
\end{lemma}
\begin{proof}
Recall from Lemma~\ref{lemma-delta-star-s-t} that
\[
   \Delta\big((\lambda\unit-T)^*(\lambda\unit-T)+t^2\unit\big)=
    \frac{|\lambda|^2}{|\lambda|^2 + (s(|\lambda|,t)-t)^2} \Delta\big(T\cc
    T + s(|\lambda|,t)^2\unit\big).
    \]
Taking logarithmic, we obtain
\[
\log\Delta\big((\lambda\unit-T)\cc(\lambda\unit-T)+t^2\unit\big)=
   \log\Delta\big(T\cc
    T + s(|\lambda|,t)^2\unit\big) + \log  \frac{|\lambda|^2}{|\lambda|^2 + (s(|\lambda|,t)-t)^2},
\]
which is
\begin{align*}
  \tau\big( \log[(\lambda \unit-T)^* & (\lambda \unit-T)+t^2\unit]\big)\\
    &=\tau\big( 
    \log(T^*T+s^2\unit)\big)+\log |\lambda|^2-\log\big( |\lambda|^2+(s-t)^2 \big).
\end{align*}
The function $s(|\lambda|,t)$ is analytic for $(|\lambda|, t)$. 
Take the derivative $\frac{\partial}{\partial \overline{\lambda}}$, it follows that 
\begin{align*}
 &\tau\big( (\lambda\unit-T)[(\lambda \unit-T)^*(\lambda \unit-T)+t^2\unit]^{-1}   \big)\\
    =&2s\cdot\tau\big( (T^*T+s^2\unit)^{-1} \big)\cdot \frac{\partial s(|\lambda|,t)}{\partial \overline{\lambda}}
    +\frac{\lambda}{|\lambda|^2}-\frac{1}{|\lambda|^2+(s-t)^2}\left( \lambda+2(s-t)\cdot\frac{\partial s(|\lambda|,t)}{\partial \overline{\lambda}} \right)\\
    =&2\left( h(s)-\frac{s-t}{|\lambda|^2+(s-t)^2}\right)\frac{\partial s(|\lambda|,t)}{\partial \overline{\lambda}}+\frac{\lambda}{|\lambda|^2}-\frac{\lambda}{|\lambda|^2+(s-t)^2}\\
    =&\frac{\lambda}{|\lambda|^2}\frac{(s-t)^2}{|\lambda|^2+(s-t)^2},
\end{align*}
where we used the identity \eqref{eqn:fixed-point-s-t} in the last step. 
The other formula can be obtained by taking the derivative with respect to $\lambda$. 
\end{proof}

\begin{remark} 
	\label{lemma:4.2}
	By applying the identity $x(x^*x+\varepsilon)^{-1}=(xx^*+\varepsilon)^{-1}x$ for $x\in\mathcal{M}$ and the tracial property, we can deduce  
 \begin{equation*}
 	 \begin{aligned}
 	 	&\tau\huge( (\lambda\unit-T)^*[(\lambda\unit-T)^*(\lambda\unit-T)+t^2\unit]^{-1} \huge)\\
 	 &\qquad\qquad=\tau\huge( (\lambda\unit-T)^*[(\lambda\unit-T)(\lambda\unit-T)^*+t^2\unit]^{-1} \huge).
 	 \end{aligned}
 \end{equation*}
\end{remark}

\begin{lemma}
	\label{lemma:formula-derivative-t-0}
We have the partial derivative formulas for $\log\Delta((T-\lambda\unit)^*(T-\lambda\unit))$
\begin{align*}
  & \tau\huge( (\lambda\unit-T)[(\lambda \unit-T)^*(\lambda \unit-T)]^{-1} \huge)\\
  =&
      \begin{cases}
         0, \quad & \text{for} \quad 0< |\lambda|\leq \lambda_1(\mu_{|T|});\\
         \frac{\lambda}{|\lambda|^2}\frac{s(|\lambda|,0)^2}{s(|\lambda|,0)^2+|\lambda|^2},\quad &\text{for} \quad \lambda_1(\mu_{|T|})<|\lambda|<\lambda_2(\mu_{|T|});\\
         \frac{\lambda}{|\lambda|^2},\quad &\text{for}\quad |\lambda|\geq \lambda_2(\mu_{|T|})
      \end{cases}
\end{align*}
and
\begin{align*}
  & \tau\huge( (\lambda\unit-T)^*[(\lambda \unit-T)^*(\lambda \unit-T)]^{-1} \huge)\\
  =&
      \begin{cases}
         0, \quad & \text{for} \quad 0< |\lambda|\leq \lambda_1(\mu_{|T|});\\
         \frac{\overline{\lambda}}{|\lambda|^2}\frac{s(|\lambda|,0)^2}{s(|\lambda|,0)^2+|\lambda|^2},\quad &\text{for} \quad \lambda_1(\mu_{|T|})<|\lambda|<\lambda_2(\mu_{|T|});\\
         \frac{\overline{\lambda}}{|\lambda|^2},\quad &\text{for}\quad |\lambda|\geq \lambda_2(\mu_{|T|}),
      \end{cases}
\end{align*}
where $s(|\lambda|,0)$ is defined in Definition \ref{def-for-s-t}.
\end{lemma}
\begin{proof}
If $\lambda_1(\mu_{|T|})<|\lambda|<\lambda_2(\mu_{|T|})$, the result follows from Lemma \ref{lemma-s-t-Cauchy} because $s(|\lambda|,0)\in (0,\infty)$ satisfying the same equation $k(s,0)=|\lambda|^2$ and is real analytic function of $\lambda$ by Lemma \ref{lemma-limit-s-0}. 	
	
If $0< |\lambda|\leq \lambda_1(\mu_{|T|})$, by Theorem \ref{thm5.15}, we have 
$\Delta(T-\lambda\unit)=\Delta(T)$. Hence, the derivatives are equal to zero. Finally, when $|\lambda|\geq \lambda_2(\mu_{|T|})$, we have $\Delta(T-\lambda\unit)=|\lambda|$ which yields the other case of the formula. 
\end{proof}

We can now give an alternative proof of Haagerup--Schultz's result on the Brown measure of $R$-diagonal operators inspired by the proof presented in \cite{BSS2018}. 
\begin{proof}[Proof of Theorem \ref{thm-main}]
For $\lambda$ such that $\lambda_1(\mu)<|\lambda|<\lambda_2(\mu)$, taking the limit
$t\rightarrow 0^+$ in \eqref{eqn:fixed-point-s-t}, we obtain
\[
  s(|\lambda|,0)^2+|\lambda|^2=\frac{s(|\lambda|,0)}{h(s(|\lambda|,0))},
\]
which yields, 
\begin{align*}
  \tau\huge( (\lambda\unit-T)^*[(\lambda \unit-T)^*(\lambda \unit-T)]^{-1} \huge)&=
   \frac{\overline{\lambda}}{|\lambda|^2}\frac{s(|\lambda|,0)^2}{s(|\lambda|,0)^2+|\lambda|^2}\\
   &=\frac{\overline{\lambda}}{|\lambda|^2} \frac{h(s(|\lambda|,0))}{s(|\lambda|,0)}s(|\lambda|,0)^2\\
   &=\frac{\overline{\lambda}}{|\lambda|^2} \tau\Big([T^*T+s(|\lambda|,0)^2\unit]^{-1}\Big)s(|\lambda|,0)^2\\
   &=\frac{\overline{\lambda}}{|\lambda|^2} \int_0^\infty \frac{s(|\lambda|,0)^2}{s(|\lambda|,0)^2+u^2}d\mu_{|T|}(u)\\
   &=\frac{\overline{\lambda}}{|\lambda|^2} \left( 1+ \psi_{\mu_{T^*T}}\left(-\frac{1}{s(|\lambda|,0 )^2}\right) \right).
\end{align*}
We can check that
\begin{equation}\label{eqn:S-inverse-psi}
  S_{\mu_{T^*T}}^{\langle -1\rangle}\left(\frac{1}{|\lambda|^{2}}\right)=\psi_{\mu_{T^*T}}\left(-\frac{1}{s(|\lambda|,0)^2}\right).
\end{equation}
Indeed,
recall that the $S$-transform satisfies $S_\mu(\psi_\mu(z))=z\cdot\frac{\psi_\mu(z)+1}{\psi_\mu(z)}$. The above calculation shows that
\begin{align*}
 \psi_{\mu_{T^*T}}\left(-\frac{1}{s(|\lambda|,0 )^2}\right) =
   \frac{s(|\lambda|,0)^2}{s(|\lambda|,0)^2+|\lambda|^2}-1=-\frac{|\lambda|^2}{s(|\lambda|,0)^2+|\lambda|^2}.
\end{align*}
Hence, 
\begin{align*}
  S_{\mu_{T^*T}} \left(\psi_{\mu_{T^*T}}\left(-\frac{1}{s(|\lambda|,0 )^2}\right)  \right)
     =-\frac{1}{s(|\lambda|,t)^2}\cdot-\frac{s(|\lambda|,t)^2}{|\lambda|^2} =\frac{1}{|\lambda|^2}
\end{align*}
which yields \eqref{eqn:S-inverse-psi}.

Therefore, we have 
\begin{align*}
&\tau\huge( (\lambda\unit-T)^*[(\lambda \unit-T)^*(\lambda \unit-T)]^{-1} \huge)\\
   =&\begin{cases}
     0,\quad &\text{for}\quad |\lambda|\leq \lambda_1(\mu);\\
      \frac{1}{\lambda} \left( 1+ \psi_{\mu_{T^*T}}\left(-\frac{1}{s(|\lambda|,0 )^2}\right) \right), \quad & \text{for} \quad \lambda_1(\mu) < |\lambda|<\lambda_2(\mu)\\
   \frac{1}{\lambda}, \quad & \text{for} \quad |\lambda|\geq \lambda_2(\mu).\\
   \end{cases}\\
   =&\begin{cases}
     0,\quad &\text{for}\quad |\lambda|\leq \lambda_1(\mu);\\
      \frac{1}{\lambda} \left( 1+ S_{\mu_{T^*T}}^{\langle -1\rangle}\left(\frac{1}{|\lambda|^{2}}\right) \right), \quad & \text{for} \quad \lambda_1(\mu) < |\lambda|<\lambda_2(\mu)\\
   \frac{1}{\lambda}, \quad & \text{for} \quad |\lambda|\geq \lambda_2(\mu).
   \end{cases}
\end{align*}
Recall that
\[
   \mu_T=\frac{2}{\pi}\frac{\partial}{\partial \overline{\lambda}}
      \frac{\partial}{\partial \lambda}\log \Delta(T-\lambda\unit)
\]
and
\[
   \frac{\partial}{\partial \lambda}\log \Delta(T-\lambda\unit)=\frac{1}{2} \tau\huge( (\lambda\unit-T)^*[(\lambda \unit-T)^*(\lambda \unit-T)]^{-1} \huge).
\]
Hence, the Brown measure $\mu_T$ is supported in $\{\lambda\in\mathbb{C}:\lambda_1(\mu_{|T|})\leq |\lambda|<\lambda_2(\mu_{|T|})\}$.
When $\lambda_1(\mu_{|T|})<|\lambda|<\lambda_2(\mu_{|T|})$, the Brown measure $\mu_T$ has the density 
\[
-\frac{1}{2\pi |\lambda|^4}\left( S_{\mu_{T^*T}}^{\langle -1\rangle}\right)'\left(\frac{1}{|\lambda|^{2}}\right).
\]
Proposition \ref{prop:s-trans-monotonicity} tells us that the map $r\mapsto S_{\mu_{T^*T}}^{\langle -1 \rangle }(r^{-2})$ is increasing on $(\lambda_1(\mu_{|T|}), \lambda_2(\mu_{|T|}))$ with range $(\delta-1,0)$, where $\delta=\mu(\{0\})$. 
This allows us to conclude that the Brown measure $\mu_{|T|}$ is the rotationally invariant probability measure such that
  \begin{align}
  \label{eqn:Brown-Rdiag-1}
    \mu_{T}\{ \lambda\in\mathbb{C}: |\lambda|\leq r\}=
      \begin{cases}
         0, \quad & \text{for} \quad  r<\lambda_1(\mu_{|T|});\\
         1+S^{\langle -1 \rangle}_{\mu_{T^*T}}(r^{-2}),\quad &\text{for} \quad \lambda_1(\mu_{|T|})<r<\lambda_2(\mu_{|T|});\\
         1,\quad &\text{for}\quad r\geq \lambda_2(\mu_{|T|}).
      \end{cases}
  \end{align}
In particular, when $\delta>0$, then $\lambda_1(\mu)=0$ and the Brown measure $\mu_T$ has a mass at the origin with the mass $\mu_T(\{0\})=\delta$. 
\end{proof}

We can now restate the Brown measure of $R$-diagonal operators in terms of subordination functions as follows. 
\begin{theorem}
\label{thm:Brown-Rdiag-main-2}
Let $T\in\CMD$ be $R$--diagonal and $\lambda\in\mathbb{C}\backslash\{0\}$. The Brown measure $\mu_{T}$ is the rotationally invariant probability measure such that
  \begin{align}
  \label{eqn:Brown-Rdiag-2}
    \mu_{T}\{ \lambda\in\mathbb{C}: |\lambda|\leq r\}=
    \frac{s(r,0)^2}{s(r,0)^2+r^2},
  \end{align}
where $s(|\lambda|,0)$ is defined in Definition \ref{def-for-s-t}. When $\lambda_1(\mu_{|T|})<r<\lambda_2(\mu_{|T|})$, the function $s(r,0)$ is the unique positive solution $s$ of the following equation
\[
     s^2+r^2-\frac{s}{h(s)}=0.
\]
\end{theorem}
\begin{proof}
When $\lambda_1(\mu_{|T|})<r<\lambda_2(\mu_{|T|})$, we recall that $s(r,0)$ denotes the unique positive solution of the equation $k(s,0)=r^2$ as shown in Definition \eqref{def-for-s-t}. The equation $k(s,0)=r^2$ can be rewritten as
\[
     s^2+|\lambda|^2-\frac{s}{h(s)}=0
\]
or
\[
   s h(s)=\frac{s^2}{s^2+r^2}.
\]
Hence, from \eqref{eqn:S-inverse-psi} in the above proof of Theorem \ref{thm-main}, we have 
\begin{align*}
 1+S^{\langle -1 \rangle}_{\mu_{T^*T}}\left(\frac{1}{r^2}\right)&=1+\psi_{\mu_{T^*T}}\left(-\frac{1}{s(r,0)^2}\right)\\
  &=\int_0^\infty\frac{s(r,0)^2}{s(r,0)^2+u^2}d\mu_{|T|}(u)\\
  &=s(r,0) h(s(r,0))\\
  &= \frac{s(r,0)^2}{s(r,0)^2+r^2}.
\end{align*}
When $0\leq r\leq \lambda_1(\mu_{|T|})$, $s(r,0)=0$; and when $r\geq \lambda_2(\mu_{|T|})$, $s(r,0)=\infty$. Hence, the formulation claimed follows from \eqref{eqn:Brown-Rdiag-1}.
\end{proof}

\begin{theorem}
	Let $T\in\CMD$ be $R$--diagonal and $\lambda\in\mathbb{C}\backslash\{0\}$, we have 
	\begin{align*}
	& \tau\huge( (\lambda\unit-T)[(\lambda \unit-T)^*(\lambda \unit-T)]^{-1} \huge)\\
	=&
	\begin{cases}
	0, \quad & \text{for} \quad 0< |\lambda|\leq \lambda_1(\mu_{|T|});\\
	\frac{\lambda}{|\lambda|^2}\mu_{T}\{ \lambda\in\mathbb{C}: |\lambda|\leq r\},\quad &\text{for} \quad \lambda_1(\mu_{|T|})<|\lambda|<\lambda_2(\mu_{|T|});\\
	\frac{\lambda}{|\lambda|^2},\quad &\text{for}\quad |\lambda|\geq \lambda_2(\mu_{|T|})
	\end{cases}
	\end{align*}
	and
	\begin{align*}
	& \tau\huge( (\lambda\unit-T)^*[(\lambda \unit-T)^*(\lambda \unit-T)]^{-1} \huge)\\
	=&
	\begin{cases}
	0, \quad & \text{for} \quad 0< |\lambda|\leq \lambda_1(\mu_{|T|});\\
	\frac{\overline{\lambda}}{|\lambda|^2}\mu_{T}\{ \lambda\in\mathbb{C}: |\lambda|\leq r\},\quad &\text{for} \quad \lambda_1(\mu_{|T|})<|\lambda|<\lambda_2(\mu_{|T|});\\
	\frac{\overline{\lambda}}{|\lambda|^2},\quad &\text{for}\quad |\lambda|\geq \lambda_2(\mu_{|T|}),
	\end{cases}
	\end{align*}
	where $s(|\lambda|,0)$ is defined in Definition \ref{def-for-s-t}.
\end{theorem}
\begin{proof}
It follows directly from the partial derivative formula in Lemma \ref{lemma:formula-derivative-t-0} and Theorem \ref{thm:Brown-Rdiag-main-2}.
\end{proof}

\section{An alternative proof for the Brown measure formula}
\label{section-alternative-proof}

We present another proof for the Brown measure of $R$-diagonal operators in this section. 
\begin{lemma}\cite[Chapter II-Theorem 1.2]{SaffTotik1997book}
	\label{lemma:log-potential}
	Let $\mu$ be a finite positive measure of compact support on the plane. Denote
	\[
	U^\mu(\lambda)=\int_\mathbb{C} \log\frac{1}{|\lambda-z|} d\mu(z), \qquad \lambda\in\mathbb{C},
	\]
	and, for any $z_0$ and $r>0$, denote the mean value
	\[
	L(U^\mu; z_0, r)=\frac{1}{2\pi}\int_{-\pi}^\pi U^\mu(z_0+re^{i\theta})d\theta.
	\]
	If $r$ is a value for which $\frac{d}{dr}L(U^\mu; z_0, r)$ exists, then 
	\[
	\mu(\{\lambda: |\lambda-z_0|\leq r \})=-r\frac{d}{dr}L(U^\mu;z_0, r).
	\]
\end{lemma}

\begin{remark}
If $U^\mu$ in Lemma \ref{lemma:log-potential} is a function of $\lambda$ smooth enough, for $z_0=0$ and $\lambda=x+iy$, we can prove Lemma \ref{lemma:log-potential} as follows
\begin{align*}
  	-\frac{d}{dr}L(U^\mu; 0, r)&=-\frac{1}{2\pi}\frac{d}{dr}\int_{-\pi}^\pi U^\mu(re^{i\theta})d\theta\\
  	&=-\frac{1}{2\pi r}\int_{|\lambda|=r}\bigg(\frac{\partial U^\mu}{\partial x}dy-\frac{\partial U^\mu}{\partial y}dx \bigg)\\
  	&=\frac{1}{2\pi r}\int_{|\lambda|=r}\nabla^2_\lambda \big(\int_\mathbb{C} \log{|\lambda-z|} d\mu(z)\big)dxdy\\
  	&=\frac{1}{r}\mu(\{\lambda:|\lambda|\leq r\}).
\end{align*}
When applying Lemma \ref{lemma:log-potential} to calculate the Brown measure of a $R$-diagonal operator, the potential we deal with is $  U^{\mu_T}(\lambda)$, which can be expressed as
\[
  U^{\mu_T}(\lambda)=-\log\Delta (T-\lambda\unit).
\]
As we know from Theorem \ref{thm5.15} and regularity result of $s(|\lambda|,0)$ in Lemma \ref{lemma-limit-s-0}, the potential is a smooth function of $\lambda$ within the domain  $\lambda_1(\mu_{|T|})<|\lambda|<\lambda_2(\mu_{|T|})$.
\end{remark}

\begin{proposition}
	Let $r>0$ and let $s(r,0)$ be as in Definition \ref{def-for-s-t} so that $s(r,0)$ is the unique solution of equation \eqref{eqn:fixed-point-s-t}
	\[
	s^2+r^2-\frac{s}{h(s)}=0,
	\]
	where $h(s)=s\phi((T^*T+s^2\unit)^{-1})$. Then, the Brown measure of $T$ is invariant under rotation and it is given by
	\begin{equation}
	\label{eqn:density-Rdig-sub-2}
	\mu_T(\{\lambda: |\lambda|\leq r \})=\frac{s(r,0)^2}{s(r,0)^2+r^2},
	\end{equation}
	for any $\lambda_1(\mu_{|T|})<r<\lambda_2(\mu_{|T|})$. If $\mu_{|T|}(\{0\})=\sigma>0$, then
	the Brown measure $\mu_T$ has a mass $\sigma$ at the origin. 
\end{proposition}
\begin{proof}
	By Proposition \ref{prop:s(r,0)-increasing}, the function $r\rightarrow s(r,0)$ is a real analytic function
	for $\lambda_1(\mu_{|T|})<r<\lambda_2(\mu_{|T|})$. Recall the Fuglede--Kadison formula \eqref{eqn:det-general-t-0}
	\[
	\Delta(T-\lambda\unit)=
	\Bigg(\frac{r^2}{r^2+s(r,0)^2}\,
	\Delta(T^* T +s(r,0)^2\unit)\Bigg)^\frac12,
	\]
	where $r=|\lambda|$. We then take the logarithmic derivative of $\Delta(T-\lambda\unit)$. 
	Note that 
	$s(r,0)\phi( (T^*T+s(r,0)^2\unit)^{-1})=h(s(r,0))$. 
	We then obtain, for $\lambda_1(\mu_{|T|})<r<\lambda_2(\mu_{|T|})$,
	\begin{align}
	\nonumber
	\frac{\partial \log\Delta(T-\lambda\unit)}{\partial r}
	& =\frac{1}{r}-\frac{1}{r^2+s(r,0)^2}\left( r +s(r,0)\frac{ds(r,0)}{dr} \right)
	+h(s(r,0))\frac{ds(r,0)}{dr}\\
	&=\frac{1}{r}-\frac{r}{r^2+s(r,0)^2} \label{eqn:derivation-log-det-Rdiag}\\
	&=\frac{1}{r}\frac{s(r,0)^2}{s(r,0)^2+r^2},\nonumber
	\end{align}
	where we used \eqref{eqn:fixed-point-s-t} to deduce the second identity.
	Now, recall that the Brown measure $\mu_T$ is the unique measure satisfying
	\begin{equation}
	\label{eqn:defn-log-det}
	\log\Delta(T-\lambda\unit)=\int_\mathbb{C}\log|\lambda-z|d\mu_T(z), \qquad \lambda\in \mathbb{C}.
	\end{equation}
	By adopting the notation as in Lemma \ref{lemma:log-potential}, denote
	\[
	U^{\mu_T}(\lambda)=-\int_\mathbb{C} \log|\lambda-z|d\mu_T(z).
	\]
	Observe that $T$ and $e^{i\theta}T$ have the same distribution. Hence, the Brown measure $\mu_T$ is invariant under the transformation $\lambda\mapsto \lambda e^{i\theta}$. The value of the function $\lambda\mapsto U^{\mu_T}(\lambda)$ 
	is also invariant if we replace $\lambda$ by $\lambda e^{i\theta}$ for any $\theta\in [-\pi,\pi]$. The Brown measure $\mu_T$ is then determined by
	\[
	\mu_T(\{\lambda: |\lambda-z_0|\leq r \})=-r\frac{\partial}{\partial r}U^{\mu_T}(\lambda)
	=r\frac{\partial \log(\Delta(T-\lambda\unit))}{\partial r}=\frac{s(r,0)^2}{s(r,0)^2+r^2}.
	\]
	This finishes the proof.
\end{proof}

We can then repeat the calculation in the Proof of Theorem \ref{thm-main} to rewrite \eqref{eqn:density-Rdig-sub-2} in terms of $S$-transform of $T\cc T$, which recovers the Brown measure formula in Theorem \ref{thm-main}.

\section{Freeness with amalgamation and $R$-diagonal operators}
We now describe the connection between the approach in previous section with the Hermitian reduction and subordination method used in \cite{BSS2018}.
We denote by $\mathcal D$ the
commutative $C^*$-algebra of diagonal matrices in $M_2(\mathbb{C})$, where we
consider the same inclusion of $M_2(\mathbb C)$ in $M_2(\mathcal M)$. Let $T\in \mathcal{M}$.
Recall from \cite{NSS2001-Rdiag} that an element $T$  in a tracial $W^*$-noncommutative
probability space $(\mathcal {M}, \tau)$ is $R$-diagonal if and only if the matrix-valued operator \begin{tiny}$\begin{bmatrix}0&T\\
	T^*&0\end{bmatrix}$\end{tiny} and $M_2(\mathbb C)$ are free with amalgamation over
$\mathcal D$ with respect to the expectation
$$\mathbb E_\mathcal D\colon M_2(\mathcal M)\to
\mathcal D,\qquad
\mathbb E_\mathcal D\begin{bmatrix}a_{11}&a_{12}\\
a_{21}&a_{22}\end{bmatrix}=\begin{bmatrix}\tau(a_{11})&0\\
0&\tau(a_{22})\end{bmatrix}.$$ 
Fix now a $\lambda\in\mathbb{C}$. The operators 
\begin{tiny}$\begin{bmatrix}0& \lambda\\
	\overline{\lambda}&0\end{bmatrix}$\end{tiny} and 
\begin{tiny}$\begin{bmatrix}0&T\\
	T^*&0\end{bmatrix}$\end{tiny} 
are free in the operator-valued noncommutative probability space
$(M_2(\mathcal{M}), \mathbb{E}_\mathcal{D})$.

Denote $X=\tiny\begin{bmatrix}0&\lambda\\
\overline{\lambda}&0\end{bmatrix}$ and by $\tiny{\mathcal D\left\langle X \right\rangle}$ the $*$-algebra generated by $\mathcal{D}$ and $X$.  
We now study the other operator-valued subordination function $\Omega_2$ such that
\begin{equation}\label{R-dia-2}
\mathbb E_{\mathcal{D}\langle X\rangle}
\left[\left(\begin{bmatrix}i\varepsilon &\lambda\\
\overline{\lambda}&i\varepsilon\end{bmatrix}-\begin{bmatrix}0&T\\
T\cc&0\end{bmatrix}\right)^{-1}\right]
=\left( \Omega_2\left(\begin{bmatrix}
i\varepsilon & 0 \\
0 & i\varepsilon
\end{bmatrix}\right) -\begin{bmatrix}0 & \lambda  \\
\overline{\lambda} & 0\end{bmatrix}\right)^{-1}.
\end{equation}
Note that this algebra $\tiny{\mathcal D\left\langle X\right\rangle}=M_2(\mathbb{C})$. 
Hence, $\mathbb E_{\mathcal{D}\langle X\rangle}$ is actually the conditional expectation onto $M_2(\mathbb{C})$ given by evaluation of the trace $\tau$ on the entries of the matrix which is the unique conditional
expectation onto $M_2(\mathbb C)$ preserving the trace.
Similar to $\Omega_1$, the subordination function $\Omega_2$ also takes values in $\mathcal{D}$. We then set 
\begin{equation}
\label{eqn:sub-operator-valued-2}
\begin{bmatrix}
i\delta_1 & 0 \\
0 & i\delta_2
\end{bmatrix} =
\Omega_2\left(\begin{bmatrix}
i\varepsilon & 0 \\
0 & i\varepsilon
\end{bmatrix}\right),
\end{equation}
where $\delta_i=\delta_i(\lambda,\varepsilon) (\text{for}\,\, i=1, 2)$ are functions of $\lambda$ and $\varepsilon$.  
Set $A_\lambda=\lambda\unit -T$. 
The left hand side of \eqref{R-dia-2} (see \cite[formula (18)]{BSS2018}) is given by 
\begin{equation}
\label{eqn:left-hand-side-2}
{
\begin{aligned}
&\mathbb E_{ \mathcal{D}\langle X \rangle}
\left[\left(\begin{bmatrix}i\varepsilon &\lambda\\
\overline{\lambda}&i\varepsilon\end{bmatrix}-\begin{bmatrix}0&T\\
T\cc&0\end{bmatrix}\right)^{-1}  \right] \\
=& \mathbb E_{ \mathcal{D}\langle X \rangle} \begin{bmatrix}
-i\varepsilon( [A_\lambda A_\lambda^*+\varepsilon^2\unit]^{-1} ) 
& A_\lambda( [A_\lambda^* A_\lambda+\varepsilon^2\unit]^{-1} ) \\
A_\lambda^*( [A_\lambda A_\lambda^*+\varepsilon^2\unit]^{-1} ) 
&  -i\varepsilon( [A_\lambda^* A_\lambda+\varepsilon^2\unit]^{-1} )
\end{bmatrix}\\
 =&\begin{bmatrix}
 -i\varepsilon \tau( [A_\lambda A_\lambda^*+\varepsilon^2\unit]^{-1} ) 
 & \tau\left(A_\lambda( [A_\lambda^* A_\lambda+\varepsilon^2\unit]^{-1} ) \right) \\
 \tau\left(A_\lambda^*( [A_\lambda A_\lambda^*+\varepsilon^2\unit]^{-1} ) \right)
 &  -i\varepsilon\tau( [A_\lambda^* A_\lambda+\varepsilon^2\unit]^{-1} )
 \end{bmatrix}
\end{aligned}
}
\end{equation}
Then the inverse matrix in the right hand side of \eqref{R-dia-2} is given by
\begin{equation}
\label{eqn:right-hand-side-22}
\begin{aligned}
\begin{bmatrix} i\delta_1 &\lambda\\
\overline{\lambda}&\delta_2 \end{bmatrix}^{-1}=
\begin{bmatrix}
  -i\delta_2(|\lambda|^2+\delta_1\delta_2)^{-1}  & \lambda (|\lambda|^2+\delta_1\delta_2)^{-1} \\
    \overline{\lambda} (|\lambda|^2+\delta_1\delta_2)^{-1}    & -i\delta_1(|\lambda|^2+\delta_1\delta_2)^{-1}
\end{bmatrix}.
\end{aligned}
\end{equation}
Traciality of $\tau$ implies that the $(1,1)$ and $(2,2)$ entries of the above matrices \eqref{eqn:left-hand-side-2} and \eqref{eqn:right-hand-side-22} are
equal, and thus $\delta_1(\lambda,\varepsilon)=\delta_2(\lambda,\varepsilon)$ and is denote by $\delta=\delta(\lambda,\varepsilon)$. Since the above matrices must be equal to each other due to \eqref{R-dia-2}, we then have 
\begin{equation}
\label{eqn:sub-Rdiag-operator-entries}
	\begin{aligned}
	\frac{\delta}{\delta^2+|\lambda|^2}&=\varepsilon \tau\Big( [(\lambda\unit-T)(\lambda\unit-T)^*+\varepsilon^2\unit]^{-1} \Big);\\
	\frac{\lambda}{ \delta^2+ |\lambda|^2}&=   \tau\left((\lambda\unit-T)\left[(\lambda\unit-T)\cc(\lambda\unit-T)+\varepsilon^2\unit\right]^{-1}\right);\\
	\frac{\overline{\lambda}}{\delta^2+|\lambda|^2}&=
	\tau\left((\lambda\unit-T)\cc\left[(\lambda\unit-T)(\lambda\unit-T)\cc+\varepsilon^2\unit\right]^{-1}\right).
	\end{aligned}
\end{equation}

We now use formulas in Lemma \ref{lemma-s-t-Cauchy} to solve $\delta$ in \eqref{eqn:sub-Rdiag-operator-entries}. We identify the function $\delta$ as the subordination studied in Section \ref{section:sub-HS}.
\begin{proposition}
	\label{prop:sub-connection-scalar-op}
For any $\varepsilon>0$, we have $\delta(\lambda,\varepsilon)=\omega_2(i\varepsilon)$  where $\omega_2$ is the subordination function such that $G_{\tilde{\mu}_{|T-\lambda\unit|}}(z)= 
G_{\nu}(\omega_2(z))$ defined in Lemma \ref{lemma-limi-omega-2}. In particular, for $t>0$, $\delta=\delta(\lambda,\varepsilon)$ is positive. Moreover,
\[
\delta=\Im \omega_2(i \varepsilon)=\frac{|\lambda|^2}{s(|\lambda|,\varepsilon)-\varepsilon}
\]
where $s(|\lambda|,\varepsilon)$ is given in Definition \ref{defn:s-lambda-t}. 
\end{proposition}
\begin{proof}
Recall that $\mu=\tilde{\mu}_{|T|}$, $\nu=\frac{1}{2}\big(\delta_{|\lambda|}+\delta_{-|\lambda|} \big)$ and $\tilde{\mu}_{|T-\lambda\unit|}=\mu\boxplus\nu$. We have $G_\nu(z)=\frac{-z}{-z^2+|\lambda|^2}$ and by \eqref{eqn:G_tilde_general}, we have 
\[
 G_{\tilde{\mu}_{|T-\lambda\unit|}}(i\varepsilon)=-i\varepsilon \tau\Big( [(\lambda\unit-T)(\lambda\unit-T)^*+\varepsilon^2\unit]^{-1} \Big)
\] 
The subordination relation $G_{\tilde{\mu}_{|T-\lambda\unit|}}(i\varepsilon)= 
G_{\nu}(\omega_2(i\varepsilon))$ is equivalent to
\begin{align*}
	\varepsilon \tau\Big( [(\lambda\unit-T)(\lambda\unit-T)^*+\varepsilon^2\unit]^{-1} \Big)=\frac{-\omega_2(i\varepsilon)}{-\omega_2(i\varepsilon)^2+|\lambda|^2}.
\end{align*}
By comparing the above identity with \eqref{eqn:sub-Rdiag-operator-entries}, we deduce that $\omega_2(\i\varepsilon)=i\delta$. Hence, by Lemma \ref{lemma-limi-omega-2}, it follows that $\delta=\delta(\lambda,\varepsilon)>0$ and 
\[
   \delta=\Im\omega_2(i\varepsilon)=\frac{|\lambda|^2}{s(|\lambda|,\varepsilon)-\varepsilon}. 
\]
This finishes the proof. 
\end{proof}

We have another proof for the partial derivative formula in Lemma \ref{lemma-s-t-Cauchy} by assuming that the $R$-diagonal operator $T$ is bounded. 
\begin{proof}[Second proof for Lemma \ref{lemma-s-t-Cauchy}]
By \eqref{eqn:sub-Rdiag-operator-entries} and Proposition \ref{prop:sub-connection-scalar-op}, we have 
\begin{align*}
	 \tau\left((\lambda\unit-T)\left[(\lambda\unit-T)^*(\lambda\unit-T)+\varepsilon^2\unit\right]^{-1}\right)
	  =\frac{\lambda}{ \delta^2+ |\lambda|^2}=\frac{\lambda}{|\lambda|}\frac{(s-\varepsilon)^2}{|\lambda|^2+(s-\varepsilon)^2}
\end{align*}
 where $s=s(\lambda,\varepsilon)$. The other partial derivative formula follows in the same way by noticing Lemma \ref{lemma:4.2}. 
\end{proof}

\begin{remark}
One can deduce Lemma \ref{lemma:formula-derivative-t-0} from Lemma \ref{lemma-s-t-Cauchy} by letting $t\rightarrow 0$. Then apply the proof for Theorem \ref{thm-main} in Section \ref{section:Brown-Rdiag-proof-1} to recover Haagerup--Larsen--Schultz's density formula for Brown measure of $R$-diagonal operators. This is exactly the approach used in \cite[Section 5]{BSS2018}. 
\end{remark}

\section{Some miscellaneous applications}

The properties of subordination functions studied in previous sections play an important role for proving convergence of empirical spectral distribution of random matrices in single ring theorem and its extension \cite{Guionnet2011-singlering,BaoEL-local-ring}. The techniques from Section \ref{section:sub-HS} could simplify various technical arguments proved in a recent work about local single ring theorem \cite[Section 7, Supplment]{BaoEL-local-ring} (in particular, Lemma 7.4, 7.5 and Lemma 7.6 in that paper).

Recall that $\mu=\tilde\mu_{|T|}$, $\nu=\frac{1}{2}\big(\delta_{|\lambda|}+\delta_{-|\lambda|}\big)$ and $\lambda_1(\mu), \lambda_2(\mu)$ as in Section \ref{section:sub-HS}, and $s(|\lambda|,t)=\Im\omega_1(it)$.
\begin{proposition}
For any $\lambda_0$ such that $\lambda_1(\mu)<|\lambda_0|<\lambda_2(\mu)$, there is a constant $C>0$ depending on $\lambda$ such that
$\Im\omega_1(0)=s(|\lambda|,0)>C$ in some neighborhood of $\lambda_0$.
\end{proposition}
\begin{proof}
It follows from the monotonicity of $k(s,t)$ in Lemma \ref{lemma:monotonicity-k-s-0} and Lemma \ref{lemma:monotoniciity-k-s-t}. 
\end{proof}

\begin{proposition}
		Let $T$ be a bounded $R$-diagonal operator.
	For any $\lambda$ such that $\lambda_1(\mu)<|\lambda|<\lambda_2(\mu)$, 
there are some positive constants $C_1, C_2$ depending on $\lambda$ such that, for all $t\geq 0$, we have 
\[
  C_1 \min\Big\{1,\frac{1}{t} \Big\}\leq \Im\omega_1(it)-t=s(|\lambda|,t)-t\leq C_2 \min\Big\{1,\frac{1}{t} \Big\}.
\]
\end{proposition}
\begin{proof}
	Recall from \eqref{eqn:s-t-for-large-t} that there is some $t_\lambda>0$ so that for any $t>t_\lambda$, we have 
	\[
	s-t=\frac{1-\sqrt{1-4|\lambda|^2h(s)^2}}{2h(s)}=\frac{2|\lambda|^2 h(s)}{1+\sqrt{1-4|\lambda|^2h(s)^2}}.
	\]
	Hence, for $t>t_\lambda$, we have 
	\[
	|\lambda|^2 h(s)<s-t<2|\lambda|^2 h(s).
	\]
	Since $h(s)=\int_\mathbb{R}\frac{s}{s^2+u^2}d\mu(u)$, for any $t>t_\lambda$, we then have 
	 \[
	    s-t<\frac{2|\lambda|^2}{s}\int_\mathbb{R}\frac{s^2}{s^2+u^2}d\mu(u)<\frac{2|\lambda|^2}{s}
	       <\frac{2|\lambda|^2}{t}.
	 \]
	As  $T$ is bounded, it follows that $\mu\boxplus\nu$ is compactly supported.
Using the identity $h(s)=h_\lambda(t)$, we then have  
\[
   s-t>|\lambda|^2h(s)=|\lambda|^2\int_\mathbb{R}\frac{t}{t^2+u^2}d\mu\boxplus\nu(t)
      >\frac{|\lambda|^2 t}{t^2+M}
\]
for some $M>0$.  It is then clear that there are $C_1, C_2$ so the desired inequality holds. 
\end{proof}

The negative moments of the resolvent of bounded $R$-diagonal operators were treated in \cite{HKS-2010-tams}. Let us illustrate an example using subordination functions to calculate the first order negative moment for any $T\in\CMD$. 
We expect that the approach can calculate higher order negative moments of unbounded $R$-diagonal operators. 

\begin{proposition}
	\label{prop:negative-moment-1st}
	Let $T\in\CMD$ be $R$--diagonal and $\lambda\in\mathbb{C}\backslash\{0\}$. Then 
	\begin{align*}
	& \tau\huge( [(T-\lambda \unit)^*(T-\lambda \unit)]^{-1} \huge)\\
	=&
	\lim_{t\rightarrow 0^+}
	\tau\huge( [(T-\lambda \unit)^*(T-\lambda \unit)+t^2\unit]^{-1} \huge)\\
	=&
	\begin{cases}
	\frac{1}{\lambda_1(\mu)^2-|\lambda|^2}, \quad & \text{for} \quad 0< |\lambda|\leq \lambda_1(\mu_{|T|});\\
	\infty,\quad &\text{for} \quad \lambda_1(\mu_{|T|})<|\lambda|<\lambda_2(\mu_{|T|});\\
	\frac{1}{|\lambda|^2-\lambda_2(\mu)^2},\quad &\text{for}\quad |\lambda|\geq \lambda_2(\mu_{|T|}).
	\end{cases} 
	\end{align*}
\end{proposition}
\begin{proof}
	Using Lemma \ref{lemma-limit-s-0}, when $0< |\lambda|\leq \lambda_1(\mu_{|T|})$ and $\lambda_1(\mu)>0$, we have
	\begin{align*}
	\tau\huge( [(T-\lambda \unit)^*(T-\lambda \unit)]^{-1} \huge)
	=&
	\lim_{t\rightarrow 0^+}
	\tau\huge( [(T-\lambda \unit)^*(T-\lambda \unit)+t^2\unit]^{-1} \huge)\\
	=&\lim_{t\rightarrow 0^+} \frac{h_\lambda(t)}{t}
	=\lim_{t\rightarrow 0^+}\frac{h(s(|\lambda|,t))}{t}\\
	=&\lim_{s\rightarrow 0^+}\frac{h(s)}{s}\lim_{t\rightarrow 0^+}\frac{s(|\lambda|,t)}{t}\\
	=&\int_0^\infty \frac{1}{u^2}d\tilde{\mu}_{|T|}\,\frac{\lambda_1(\mu_{|T|})^2}{\lambda_1(\mu_{|T|})^2)-|\lambda|^2}\\
	=&\frac{1}{\lambda_1(\mu_{|T|})^2-|\lambda|^2}.
	\end{align*}
	Similarly, when $|\lambda|\geq \lambda_2(\mu_{|T|})$, we have 
	\begin{align*}
	\tau\huge( [(T-\lambda \unit)^*(T-\lambda \unit)]^{-1} \huge)
	=&
	\lim_{t\rightarrow 0^+}
	\tau\huge( [(T-\lambda \unit)^*(T-\lambda \unit)+t^2\unit]^{-1} \huge)\\
	=& \lim_{t\rightarrow 0^+}\frac{h_\lambda(t)}{t}\\
	=&\lim_{s\rightarrow\infty}sh(s)  \lim_{t\rightarrow 0^+}\frac{1}{s(|\lambda|,t)t}\\
	=&\frac{1}{|\lambda|^2-\lambda_2(\mu_{|T|})^2)}.
	\end{align*}
	When $\lambda_1(\mu_{|T|})<|\lambda|<\lambda_2(\mu_{|T|})$, as $\lim_{t\rightarrow 0^+}s(|\lambda|,t)=s(|\lambda|,0)>0$, we have 
	\begin{align*}
	\tau\huge( [(T-\lambda \unit)^*(T-\lambda \unit)]^{-1} \huge)
	=&  \lim_{t\rightarrow 0^+}
	\tau\huge( [(T-\lambda \unit)^*(T-\lambda \unit)+t^2\unit]^{-1} \huge)\\
	&=\lim_{t \rightarrow 0^+}\frac{h(s(|\lambda|,t)}{s(|\lambda|,t)}\frac{s(|\lambda|,t)}{t}
	=\infty.
	\end{align*}	
This finishes the proof. 
\end{proof}

\begin{example}
Let $c_t$ be Voiculescu's circular operator with variance $t$, and $g_t$ be a semicircular element of variance $t$,  and let $u$ be a Haar unitary operator that is $*$-free from $\{c_t, g_t\}$. In \cite{Zhong2021Brown}, the author calculated the Brown measure density of $g_t+u$. It is showed that the Brown measure of $g_t+u$ is the pushforward measure of $c_t+u$ under some natural pushforward map.  The operator $c_t+u$ is $R$-diagonal and the outer radii of the support of the Brown measure of $c_t+u$ is $\sqrt{t+1}$ and the inner radii is $\sqrt{(1-t)_+}$. Indeed, to calculate the inner radii $\lambda_1(c_t+u)=1/\tau[(c_t+u)^*(c_t+u)]$, we observe that $|c_t+u|  \underset{\ast\CD}{\sim} |c_t+ \unit|$ and then apply Proposition \ref{prop:negative-moment-1st}. 
\end{example}

\bigskip
\noindent
{\bf Acknowledgment.} The author wants to thank Serban Belinschi, Hari Bercovici, Kamil Szpojankowski, Alexandru Nica, and Zhi Yin for helpful discussions. The proof in Section \ref{section-alternative-proof} was found when the author was preparing his notes for an online learning seminar on Brown measure in summer 2021 organized by Professor Nica. He wants to thank seminar participants for useful comments. He is grateful to an anonymous referee and Professor Yemon Choi for very careful readings and suggestions which greatly improve the readability of this paper.

\bibliographystyle{acm}
\bibliography{BrownMeasure}



\end{document}